\documentclass[letterpaper,11pt,openany,oneside]{article}

\usepackage[T1]{fontenc}
\usepackage[latin1]{inputenc}
\usepackage{epsfig}
\usepackage{amsmath,amssymb}
\usepackage[active]{srcltx}
\usepackage{dsfont}
\usepackage{tikz}
\usetikzlibrary{shapes}

\numberwithin{equation}{section}
\newtheorem{theoreme}{Theorem}[section]
\newtheorem{proposition}{Proposition}[section]
\newtheorem{remarque}[theoreme]{Remark}

\newtheorem{lemme}[theoreme]{Lemma}
\newtheorem{corollaire}[theoreme]{Corollary}
\newtheorem{definition}[theoreme]{Definition}
\newenvironment{proof}[1][Proof]{\noindent \textbf{#1.}~ }
{\hfill\rule{2mm}{2mm} \vspace{\parskip} }

\newcommand{\R}{\ensuremath{\mathbb R}}
\newcommand{\val}{\mathsf{val}}
\newcommand{\PP}{\ensuremath{\mathbb P}}

\newcommand{\E}{\ensuremath{\mathbb E}}

\newcommand{\NN}{\ensuremath{\mathbb N}}
\newcommand{\N}{\ensuremath{\mathbb N}}

\newcommand{\ind}{\ensuremath{\mathds 1}}
\newcommand{\xx}{\ensuremath{\mathbf x}}
\newcommand{\yy}{\ensuremath{\mathbf y}}

\newcommand{\Id}{\ensuremath{\operatorname{Id}}}

\newcommand{\De}{\Delta}

\newcommand{\jj}{\ensuremath{\mathbf j}}

\newcommand{\la}{\lambda}
\newcommand{\ep}{\varepsilon}

\newcommand{\al}{\alpha}
\newcommand{\be}{\beta}
\newcommand{\si}{\sigma}

\newcommand{\ga}{\gamma}

\definecolor{darkblue}{rgb}{0,0,0.7} 
\newcommand{\darkblue}{\color{darkblue}} 
\newcommand{\defn}[1]{\emph{\darkblue #1}} 
\title{Constant payoff in absorbing games}
\author{\textsc{Miquel Oliu-Barton}\\[0.25cm]
\small Universit\'e Paris-Dauphine \\ 
\small miquel.oliu.barton@normalesup.org}
\date{February 2020}

\begin{document}
\maketitle

\textbf{Abstract}. Oliu-Barton and Ziliotto \cite{OBZ18} proved that the constant payoff property holds 
for discounted stochastic games, as conjectured by Sorin, Venel and Vigeral ~\cite{SVV10}. That is, the existence of a pair of asymptotically optimal strategies so that the average rewards are constant on any fraction of the game. That a similar property holds for stochastic games with an arbitrary evaluation of the stage rewards is still open. In this paper, we prove that the constant constant payoff property holds for a class of stochastic games 
which includes the well-known model of absorbing games and stochastic games with two states. 

\tableofcontents 
\section{Introduction}
\paragraph{Model.}
 Stochastic games were introduced by Shapley \cite{shapley53} in order to model a repeated interaction between two opponent players in a changing environment.  The game proceeds in stages. At each stage $m\in \N$ of the game, players play a zero-sum game that depends on a state variable. Formally, knowing the current state $k_m$, Player 1 chooses an action $i_m$ and Player 2 chooses an action $j_m$. Their choices occur independently and simultaneously and have two consequences: first, they produce a stage payoff $g(k_m,i_m,j_m)$ which is observed by the players and, second, they determine the law $q(k_m,i_m,j_m)$ of the next period's state $k_{m+1}$. Thus, the sequence of states follows a Markov chain controlled by the actions of both players.  To any sequence of nonnegative weights $\theta=(\theta_m)$
and any initial state $k$ corresponds the $\theta$-weighted average stochastic game 
 which is one in which Player 1 maximizes the expectation of 
$$\sum_{m \geq 1} \theta_m g(k_m,i_m,j_m)$$ given that $k_1=k$, while Player 2 minimizes this same amount. A crucial aspect in this model is that \emph{the current state is commonly observed} by the players at every stage. Another one is stationarity: the transition function and stage payoff function do not change over time. 

A $\theta$-evaluated stochastic game is thus described by a tuple $(K,I,J,g,q,k,\theta)$ where $K$ is a set of states, $I$ and $J$ are the sets of actions of both players, $g:K\times I\times J\to \R$ is the reward function, $q:K\times I\times J\to \De(K)$ is the transition function, $k$ is the initial state and $\theta$ is a normalized sequence of nonnegative weights, i.e. so that $\sum_{m\geq 1}\theta_m=1$ with no loss of generality. A $\la$-discounted stochastic game is one where $\theta_m=\lambda(1-\lambda)^{m-1}$ for all $m\ge 1$ for some $\la\in(0,1]$. A \emph{$T$-stage stochastic game} is one where $\theta_m=\frac{1}{T} \ind_{\{m\leq T\}}$ for all $m\ge 1$ for some $T\in \N$. 
Like in Shapley's seminal paper \cite{shapley53}, we assume throughout this paper that $K$, $I$, $J$ are finite sets, and identify the set $K$ with $\{1,\dots,n\}$.

\paragraph{Selected past results.}  Every stochastic games $(K,I,J,g,q,k,\theta)$ has a value, denoted by $v_\theta^k$. Although only stated for the discounted case, this result follows from Shapley~\cite{shapley53}.
We use the notation $v_\la^k$ and $v^k_T$ to refer respectively to the value of a $\la$-discounted and a $T$-stage stochastic game. Bewley and Kohlberg~\cite{BK76} proved the convergence of $v_\la^k$ as $\la$ goes to $0$ and the convergence of $v_T^k$ as $T\to +\infty$, to the same limit. 
Mertens and Neyman~\cite{MN81,MN82} proved the existence of the value $v^k$, that is, Player~1 can ensure that the average reward is at least $v^k$ in any $T$-stage stochastic game with $T$ large enough, and similarly Player~2 can ensure that the average reward is at most $v^k$. Neyman and Sorin~\cite{NS10} studied stochastic games with a random number of stages, and proved that the values converge to $v^k$ as the expected number of stages tends to $+\infty$, and the 
expected number of remaining stages decreases throughout the game. 
Ziliotto~\cite{ziliotto16} proved that $v_\theta^k$ converges to $v^k$ as $\|\theta\|:=\max_{m\geq 1}\theta_m$ goes to $0$ provided that $\sum_{m\geq 1} |\theta_{m +1}^p - \theta_m^p |$ converges to zero for some $p > 0$. 
The value was recently characterized by Attia and Oliu-Barton~\cite{AOB19}. 

\paragraph{The constant-payoff property.} 
A remarkable property, referred to as the \emph{constant-payoff property} was proved by Sorin, Venel and Vigeral \cite{SVV10} in the framework of single decision-maker problems, and conjectured to hold in any stochastic game. Their conjecture goes as follows. 
\begin{itemize}
\item \emph{The discounted case}. 
For any sufficiently small $\lambda$ there exists a pair of optimal strategy so that 
$\sum_{m=1}^M \lambda(1-\lambda)^{m-1} g(k_m,i_m,j_m)$ is approximatively equal to $(\sum_{m=1}^M\lambda(1-\lambda)^{m-1})v^k$ in expectation. 
\item \emph{The general case}. 
Because $\sum_{m\geq 1}\lambda(1-\lambda)^{m-1}=1$, Sorin~\cite{sorin02} proposed the interpretation of $\sum_{m=1}^M \lambda(1-\lambda)^{m-1}$ as the fraction of the game that has been played at stage $M$. The notion of a fraction of the game extends to any evaluation $\theta$, and the general constant-payoff conjecture is the existence of a pair of strategies so that for any sufficiently small $\|\theta\|$, 
$\sum_{m=1}^M \theta_m g(k_m,i_m,j_m)$ is approximatively equal to $(\sum_{m=1}^M\theta_m) v^k$ in expectation. 
\end{itemize}
The constant-payoff conjecture was established Oliu-Barton and Ziliotto~\cite{OBZ18} for discounted stochastic games, and the problem remains open for arbitrary evaluations of the rewards.

\paragraph{Main result.} In this paper we solve the general constant-payoff conjecture for a class of stochastic games (see Theorem~\ref{CPA2} below) which 
 includes the well-known model of absorbing games introduced by Kohlberg\cite{kohlberg74}, and stochastic games with two states. 
(It is worth noting that, for absorbing games, the discounted constant-payoff property was proved by Sorin and Vigeral~\cite{SV19} using an independent approach.)


\section{Stochastic games}
In the sequel, let $(K,I,J,g,q,k,\theta)$ denote a $\theta$-evaluated stochastic game. In order to state our results formally, we start by recalling some definitions.
\subsection{Strategies}
The sequence $(k_1,i_1,j_1,...,k_m,i_m,j_m,...)$ generated along the game is called \textit{a play}. The set of plays is $(K\times I\times J)^{\N}$.  
\begin{definition}\label{strat}
\begin{itemize}
\item[]
\item A \defn{strategy} for a player specifies a mixed action to each possible set of past observations. Formally, a strategy for Player~1 is a collection of maps $\sigma^1=(\sigma^1)_{m \geq 1}$, where $\sigma^1_m:(K \times I \times J)^{m-1} \times K \rightarrow \Delta(I)$. Similarly, a strategy for Player~2 is a collection of maps $\sigma^2=(\sigma^2)_{m \geq 1}$, where $\sigma^2_m:(K \times I \times J)^{m-1} \times K \rightarrow \Delta(J)$.
\item A \defn{stationary strategy} is one that plays according to the current state only. Formally, a stationary strategy for Player~1 is a mapping $x:K\to \De(I)$, and a stationary strategy for Player~$2$ is a mapping $y:K\to \De(J)$. 
\end{itemize}
\end{definition}
\noindent \textbf{Notation.} The sets of strategies for Player 1 and 2 are denoted by $\Sigma$ and $\mathcal{T}$, respectively, and the sets of stationary strategies by $\De(I)^n$ and $\De(J)^n$.  
For any pair $(\si,\tau) \in \Sigma\times \mathcal{T}$ we denote by $\PP^{k}_{\si,\tau }$ the unique probability measure on the set of plays $(K\times I\times J)^\N$ induced by $(\si,\tau)$, $k_1=k$ and $q$.   (Note that the dependence on the transition function $q$ is omitted). This probability is well-defined by the Kolmogorov extension theorem, and the expectation with respect to the probability $\PP^k_{\si,\tau }$ is denoted by $\E^k_{\si,\tau}$. 
\subsection{The cumulated payoffs}
For each normalized sequence of nonnegative weights $\theta$ we introduce the clock function $\varphi(\theta,\, \cdot\,):[0,1]\to \N$ by setting
$$\varphi(\theta,t):=\inf\{M\geq 1,\ \sum\nolimits_{m=1}^M \theta_m\geq t\}\qquad \forall t \in [0,1]\,.$$
The \defn{cumulated payoff at time t} is defined for any pair of strategies $(\si,\tau)\in \Sigma\times \mathcal{T}$ as 
$$\ga_\theta^k(\si,\tau;t):=\E_{\si,\tau}^k\left[\sum\nolimits_{m= 1}^{\varphi(\theta,t)} \theta_m g(k_m,i_m,j_m)\right]\,.$$
Note that the case $t=1$ corresponds to the expectation of the $\theta$-evaluation of the stage rewards. 
\subsection{Optimal and asymptotically optimal strategies}
\begin{definition} 
An \defn{optimal strategy} of Player~1 is an element $\si\in \Sigma$ so that, for all $\tau\in \mathcal{T}$,  
$$\ga_\theta^k(\si,\tau;t)\geq v_\theta^k\,.$$
An optimal strategies of Player~2 is defined in a similar way. That is, it is an element $\tau \in \Sigma$ so that, for all $\si \in \Sigma$ one has $\ga_\theta^k(\si,\tau;t)\leq v_\theta^k$.
\end{definition}
\begin{definition} Let $\ep\geq 0$.  A family of strategies $(\si_\theta^\ep)$ indexed by $\theta$ is asymptotically $\ep$-optimal for Player~1 if for any $\tau \in \mathcal{T}$, 
$$\liminf_{\la\to 0}\ga_\la^k(x_\la^\ep,\tau)\geq v^k-\ep\,.$$ 
Asymptotically $\ep$-optimal for Player~2 are defined in a symmetric way. 
\end{definition}



\subsection{The general constant payoff conjecture}
Sorin, Venel and Vigeral ~\cite{SVV10} conjectured the existence of a pair of asymptotically $0$-optimal strategies $(\si(\theta),\tau(\theta))$, indexed by $\theta$, and so that 
$$\lim_{\|\theta\|\to 0}\ga_\theta^k(\si(\theta),\tau(\theta);t)=t v^k\qquad \forall t\in [0,1]\,.$$

\paragraph{A family of asymptotically $0$-optimal strategies.}
Let $(x_\la,y_\la)$ a fixed family of optimal stationary strategies so that $\la\mapsto x^k_\la(i)$ and $\la\mapsto y_\la^k$ admit a Puiseux expansion near $0$ for all $(i,j)\in I\times J$. For each $\theta\in \De(\N)$ and $m\geq 1$ set $$\la_m=\frac{\theta_m}{\sum_{m' \geq m}  \ \theta_{m'}}\,.$$  
Define then the strategy pair $(\si^\theta,\tau^\theta)$ by setting,
\begin{equation}\label{defstrat}(\si^\theta_m,\tau^\theta_m):=(x_{\la^\theta_m},y_{\la^\theta_m})\qquad  \forall m\geq 1\,.\end{equation}
The family $(\si^\theta,\tau^\theta)$ is asymptotically $0$-optimal by Ziliotto \cite{ziliotto16}, so it is a good candidate for tackling the general constant-payoff conjecture.

\subsection{Main result}\label{mainOB15}
Consider the following symmetric conditions $(H1)$ and $(H2)$. These conditions preclude going from one state to another and back, and then to a third state, for a fixed pure stationary strategy of the opponent. 
\begin{itemize}
\item $(H1)$ There does not exist a triplet of different states $(\ell,\ell',\underline{\ell})$ and a tuple of actions $(i,\underline{i},j,i',j')$, $i\neq \underline{i}$ so that 
$$q(\ell' \,| \, \ell, i,j)q(\ell \,| \, \ell', i',j')q(\underline{\ell} \,| \, \ell, \underline{i},j)
>0\,.$$
\item $(H2)$ There does not exist a triplet of different states $(\ell,\ell',\underline{\ell})$ and a tuple of 
and a tuple of actions $(i,i',j,\underline{j}, j')$, $j\neq \underline{j}$ so that 
$$q(\ell' \,| \, \ell, i,j)q(\ell \,| \, \ell', i',j')q(\underline{\ell} \,| \, \ell, i,\underline{j})>0\,.$$
\end{itemize}


\begin{theoreme}\label{CPA2} 
Any stochastic game $(K,I,J,g,q,k,\theta)$ satisfying $(H1)$ and $(H2)$ satisfies general constant-payoff property. More precisely, the family of asymptotically $0$-optimal strategies defined in \eqref{defstrat} satisfies 
$$\lim_{\|\theta\|\to 0}\ga_\theta^k(\si^\theta,\tau^\theta;t)=t v^k \qquad \forall t\in [0,1]\,.$$
\end{theoreme}

\begin{definition} A stochastic game $(K,I,J,g,q,k,\theta)$ is \defn{absorbing} if for every state $\ell\neq k$ is absorbing, i.e. $q(\ell\,|\, \ell,i,j)=1$ for all $(i,j)\in I\times J$.
\end{definition} 
The following result is a direct consequence of Theorem~\ref{CPA2}. 
\begin{corollaire}\label{CPA} Suppose that $(K,I,J,g,q,k,\theta)$ is an absorbing game. Then,
$$\lim_{\|\theta\|\to 0}\ga_\theta^k(\si^\theta,\tau^\theta;t)=t v^k \qquad \forall t\in [0,1]\,.$$
\end{corollaire}

\section{Proofs} 
Theorem~\ref{CPA2} relies on the two following properties, which hold for all $t\in [0,1]$: 
\begin{itemize}
\item $\lim_{\la\to 0}\ga_\la^k(x_\la,y_\la;t)=t v^k$.
\item $\lim_{\|\theta\|\to 0}\ga_\theta^k(\si^\theta,\tau^\theta;t)$ exists and does not depend on the family of vanishing evaluations.
\end{itemize}

We start by recalling some useful results from \cite{JOB13, OBZ18} on discounted stochastic games. Then, for the sake of simplicity, we will establish Corollary~\ref{CPA} firs. Finally, we will extend this result to the class of stochastic games of Theorem~\ref{CPA2}.

\subsection{Preliminaries}
We recall the following results from \cite{JOB13, OBZ18} for discounted stochastic games. 
\begin{theoreme}\label{aer} For every $t\in [0,1]$ the limit 
$\Pi_t:=\lim_{\la\to 0} \sum_{m\geq 1}\la(1-\la)^{m-1}Q_\la^{m-1}\in \R^{n\times n}$ exists.
Furthermore, there exist $p\in \N$, which stands for the number of payoff-relevant cycles, $\Phi\in \R^{n\times p}$, $A\in \R^{p\times p}$ and $M\in \R^{p\times n}$
so that 
$$\Pi_t=\int_0^t \Phi e^{-\ln(1-s)A} M ds\,.$$
\end{theoreme}

\begin{corollaire} The map $t\mapsto \Pi_t$ is twice differentiable on $(0,1)$ and 
$$\frac{\partial^2}{\partial t^2} \Pi_t =\frac{1}{1-t}\Phi e^{-\ln(1-t) A} A M\qquad \forall t\in(0,1)\,.$$ 
\end{corollaire} 
\begin{lemme}\label{AMg} Suppose that $AM g_0=0$. Then, $\lim_{\la\to 0}\ga_\la(x_\la,y_\la;t)=t v$.
\end{lemme}
Indeed, in this case one has $\frac{\partial^2}{\partial t^2} \Pi_t g_0=0$. Therefore, there exists vectors $\al,\be \in \R^n$ so that $\Pi_t g_0 =\al t+\be$. The boundedness of $g_0$ implies that 
$$\lim_{\la\to 0}\ga_\la(x_\la,y_\la;t)= \Pi_t g_0\qquad \forall t\in[0,1]\,.$$
In particular one has $\Pi_0 g_0=0$ because $\Pi_0=0$ by definition, and $\Pi_1 g_0=v$ by the optimality of $(x_\la,y_\la)$. Consequently, $\al=v$ and $\be=0$.

\subsection{Absorbing games}
Throughout this section, $(K,I,J,g,q,k,\theta)$ is an \defn{absorbing game}, where all states except $k$ are absorbing. 
For any initial state $\ell\neq k$, the stochastic game $(K,I,J,g,q,\ell,\theta)$ is equivalent to the matrix game $g^\ell \in \R^{I\times J}$ so that $v_\theta^\ell=v^\ell=\val \, g^\ell$. For this reason, we restrict our attention to the game with initial state $k$. 




\subsection{The constant payoff in discounted absorbing games}
\begin{proposition} Let $(x_\la,y_\la)$ be a pair of optimal strategies that admit an expansion in Puiseux series near 0. Then, for all $t\in [0,1]$, $\lim_{\la\to 0}\ga_\la^k(x_\la,y_\la;t)=t v^k$
\end{proposition} 
Let $p\in \N$, $A\in \R^{p\times p}$ and $M\in \R^{p\times n}$ be as in 
Theorem~\ref{aer}. By Lemma~\ref{AMg} it is enough to prove that $AM g_0=0$. 
First of all, note that any absorbing state $\ell\neq k$ is payoff relevant cycle, and that $A^{\ell,\ell}=0$ because $A^{\ell,\ell}$ is the the normalized exit rate from $\ell$.  
We distinguish two cases, depending on whether $k$ is a payoff-relevant cycle or not. 

\noindent\textbf{Case 1:} $k$ is not a payoff-relevant cycle. In this case $p=n-1$ and $A=0$, so that $AMg_0=0$. 

\noindent\textbf{Case 2:} $k$ is a payoff-relevant cycle. In this case $p=n$ and $M=\Id$ so it is enough to prove $A g_0=0$. If $A^{k,k}=0$, then again $A=0$ so that $A g_0=0$. So suppose that $|A^{k,k}|>0$. This is the interesting case, as it corresponds to a situation where the absorption occurs after a positive, but random fraction of the game. We now prove this case. 

\begin{proof}
For all $(i,j)\in I\times J$, let $(c(i),e(i),c'(j), e'(j))\in \R^4_+$ so that $x_\la^k(i)=c(i)\la^{e(i)}+o(\la^{e(i)})$ and $y_\la^k(j)=c'(j)\la^{e'(j)}+o(\la^{e'(j)})$, and so that $c=0$ implies $e=0$. Introduce the sets of actions: 
$$\begin{cases} 
I_0=\{i\in I\,|\, e(i)=0\}\\
I_*=\{i\in I\,|\, e(i)\in (0,1)\}\\
I_1=\{i\in I\,|\, e(i)=1\}\\
I_+=\{i\in I\,|\, e(i)>1\}\,.
\end{cases} \quad 
\begin{cases} 
J_0=\{j\in J\,|\, e'(j)=0\}\\
J_*=\{j\in J\,|\, e'(j)\in (0,1)\}\\
J_1=\{j\in J\,|\, e'(j)=1\}\\
J_+=\{j\in J\,|\, e'(j)>1\}\,.\end{cases} $$
The sets $I_0,I_*,I_1,I_+$ partition $I$ while $J_0,J_*,J_1,J_+$ partition $J$. Because $k$ is a payoff-relevant cycle and $|A^{k,k}|>0$, the normalized exit rate from $k$ to any other state $\ell\neq k$ is given by $A^{k,\ell}$, which is the sum of the three normalized exit rates corresponding to $I_0\times J_1$, the pairs $(i,j)\in I_*\times J_*$ so that $e(i)+e'(j)=1$, and $I_1\times J_0$, denoted respectively by $A^{k,\ell}_{10}$, $A^{k,\ell}_{*}$, and $A^{k,\ell}_{01}$. Thus, $A=A_{10}+A_*+A_{01}$, where these are four $n\times n$ Markov matrices. The situation can be visualized as follows. 
\begin{center}
\begin{tikzpicture}[xscale=1, yscale=1]
\node [above] at (3.5,3+1) {$J_+$};
\node [above] at (2.5,3+1) {$J_1$};
\node [above] at (1.5,3+1) {$J_*$};
\node [above] at (0.5,3+1) {$J_0$};
\node [left] at (0,-0.5+1) {$I_+$};
\node [left] at (0,0.5+1) {$I_1$};
\node [left] at (0, 1.5+1) {$I_*$};
\node [left] at (0, 2.5+1) {$I_0$};
\node [scale=1] at (0.5, 2.5+1) {$g_0^k$};
\draw (0,0) to (0,4) to (4,4) to (4,0) to (0,0);
\draw (1,0) to (1,4);
\draw (2,0) to (2,4);
\draw (3,0) to (3,4);
\draw (0,1) to (4,1);
\draw (0,2) to (4,2);
\draw (0,3) to (4,3);
\draw [fill=gray!30] (0,0+1) to (0,1+1) to (1,1+1) to (1,0+1) to (0,0+1);
\draw [fill=gray!30] (1,1+1) to (1,2+1) to (2,2+1) to (2,1+1) to (1,1+1);
\draw [fill=gray!30] (2,2+1) to (2,3+1) to (3,3+1) to (3,2+1) to (2,2+1);
\node at (0.5,0.5+1) {$A_{10}$}; 
\node at (1.5,1.5+1) {$A_*$}; 
\node at (2.5,2.5+1) {$A_{01}$}; 
\end{tikzpicture}
 \end{center}
The actions in $I_0\times J_0$ determine the payoff $g_0^k$, while the shaded areas correspond to the pairs of actions which determine the transitions from $k$ to the set of absorbing states. The actions in $I_+$ and $J_+$ are irrelevant as they do not affect neither$g_0^k$ nor $A^{k,\ell}$ for all $\ell\neq k$. 
\end{proof}

\noindent\textbf{Case 2a:} Suppose by contradiction that $g^k_0>v^k$ and $|A^{k,k}_{01}|>0$. In this case, Player~2 can deviate from $y_\la$ to a strategy $\widetilde{y}_\la$ which changes the probabilities of playing actions in $J_1$ to $c'(j)\la^{1-\ep}$ for a sufficiently small $\ep$ (say, smaller than all nonzero $e(i)$ and $e'(j)$). By doing so, the probability that the state $k$ is left before stage $t/\la$ goes to $1$ for any $t>0$. Consequently, 
\begin{equation}\label{PCA1}\lim_{\la\to 0}\ga_\la^k(x_\la,\widetilde{y}_\la)=\sum_{\ell\neq k}\frac{A^{k,\ell}v^\ell}{|A^{k,k}|}\,.\end{equation}
On the other hand, if Player~1 deviates from $x_\la$ to a strategy $\widetilde{x}_\la$ which plays actions outside $I_0$ to 0, then the transition from $k$ to the set of absorbing states depends only on $A_{01}$, and one has 
\begin{equation}\label{PCA2}\lim_{\la\to 0}\ga_\la^k(\widetilde{x}_\la,y_\la)=\frac{g_0+\sum_{\ell\neq k} A_{01}^{k,\ell} v^\ell}{1+|A^{k,k}|_{01}}\,.\end{equation}
Yet, the optimality of $(x_\la)$ and $(y_\la)$ implies that 
\begin{equation}\label{PCA3}\lim_{\la\to 0}\ga_\la^k(\widetilde{x}_\la,y_\la)\leq v \leq \lim_{\la\to 0}\ga_\la^k(x_\la,\widetilde{y}_\la)\,.\end{equation}
The relations \eqref{PCA1}, \eqref{PCA2} and  \eqref{PCA3} are not compatible with $g_0^k>v^k$, a contradiction.  
\noindent\textbf{Case 2b:} Suppose by contradiction that $g^k_0>v^k$ and $A^{k,k}_{01}=0$. 
In this case, for the strategy $(\widetilde{x}_\la)$ described in the previous case, one has 
$$\lim_{\la\to 0}\ga_\la^k(\widetilde{x}_\la,y_\la)=g_0^k\,.$$
This contradicts the optimality of $(y_\la)$.\\

We thus conclude that $g_0^k\leq v^k$. Similarly, reversing the roles of the players one obtains $g_0^k\geq v^k$ so that $g^k_0=v^k$. Now one the one hand, $g^\ell=v^\ell$ for all $\ell\neq k$ and on the other $A^{\ell,\ell'}=0$ for all $\ell\neq k$ and $\ell'$. Therefore, $Ag=Av=0$ as soon as $A^{k,k}v^k+\sum_{\ell\neq k}A^{k,\ell}v^\ell=0$, which follows from 
\begin{equation}\label{PCA1}v^k=\lim_{\la\to 0}\ga_\la^k(x_\la,y_\la)=\frac{g^k_0+\sum_{\ell\neq k} A^{k,\ell}v^\ell}{1+|A^{k,k}|}\,.\end{equation}

\subsection{The constant payoff for general absorbing games}
\begin{proposition}\label{limittheta} Let $(\si^\theta,\tau^\theta)$ be the family of asymptotically $0$-optimal strategies defined in \eqref{defstrat}. Then, for all $t\in [0,1]$, $\lim_{\|\theta\|\to 0}\ga_\theta^k(\si^\theta,\tau^\theta;t)$ exists and does not depend on the vanishing evaluations.
\end{proposition}

We start by two technical lemmas. 

\begin{lemme}\label{tecnico1} Let $0\leq t<t+h<1$. Then, for all $e\geq 0$,
$$\lim_{\|\theta\|\to 0} \sum_{m=\varphi(t,\theta)}^{\varphi(t+h,\theta)} 
(\la_m^\theta)^e =
\begin{cases}
+\infty &  \emph{ if } e<1\\
\ln\left (1-\frac{h}{1-t}\right) & \emph{ if } e=1\\
0 &  \emph{ if }  e>1\,. 
\end{cases}
$$

\end{lemme}

\begin{proof} For each $m\geq 1$ set $t^\theta_m:=\sum_{m'=1}^{m-1}\theta_m$ so that 
$\la^\theta_m= \frac{\theta_{m+1}}{1-t^\theta_m}$. 
Then, for any $m$ between $\varphi(t,\theta)$ and $\varphi(t+h,\theta)$ one has $t\leq t^\theta_m\leq t+h$, so that 
\begin{equation}\label{lae}
\left(\frac{\theta_{m+1}}{1-t}\right)^e\ \leq \left(\frac{\theta_{m+1}}{1-t^\theta_m}\right)^e \leq \left(\frac{\theta_{m+1}}{1-t-h}\right)^e\end{equation}
We distinguish three cases, depending on whether $e=1$, $e>1$ or $e<1$.\\
\noindent \textbf{Case $e=1$.} It is enough to note that 
$\lim_{\|\theta\|\to 0}\sum_{m=\varphi(t,\theta)}^{\varphi(t+h,\theta)}\theta_{m+1}= h$, add the inequalities of \eqref{lae} for all $\varphi(\theta,t)\leq m\leq \varphi(\theta,t+h)$ and then take $\|\theta\|$ to $0$, to obtain the desired result, i.e. 
$$\frac{h}{1-t} \leq \lim_{\|\theta\|\to 0} \sum_{m=\varphi(t,\theta)}^{\varphi(t+h,\theta)}\la^\theta_m \leq \displaystyle \frac{h}{1-t-h}$$
It follows that, as $h$ tends to $0$, $\lim_{\|\theta\|\to 0} \sum_{m=\varphi(t,\theta)}^{\varphi(t+h,\theta)}\la^\theta_m=\displaystyle\frac{h}{1-t}+o(h)$, and hence the result.\\
\noindent \textbf{Case $e<1$.} In this case $\theta^e_{m+1}\geq \|\theta\| ^{e-1}\theta_{m+1}$. 
From \eqref{lae} one derives
\begin{eqnarray}\sum_{m=\varphi(t,\theta)}^{\varphi(t+h,\theta)} (\la^\theta_m)^e&\geq &  \|\theta\| ^{e-1}\sum_{m=\varphi(t,\theta)}^{\varphi(t+h,\theta)}\theta_{m+1}\,.
\end{eqnarray}
The result follows, since $\lim_{\|\theta\|\to 0} \|\theta\| ^{e-1}=+\infty$ and 
$\lim_{\|\theta\|\to 0} \sum_{m=\varphi(t,\theta)}^{\varphi(t+h,\theta)}\theta_{m+1}=h$.\\
\noindent \textbf{Case $e>1$.} In this case $\theta^e_{m+1}\leq \|\theta\| ^{e-1}\theta_{m+1}$, and the result follows from \eqref{lae} like the previous case, since $\lim_{\|\theta\|\to 0} \|\theta\| ^{e-1}=0$ in this case. \end{proof}

\begin{lemme}\label{convth} Let $(a^\theta_m)_{m\geq 1}$ be a sequence in $[0,1]$ and let
$f(\,\cdot\,,\theta):[0,1]\to \NN$ be a nondecreasing mao so that $\lim_{\|\theta\|\to 0} \sum_{m=1}^{f(t,\theta)}\theta_m=t$. 
 Suppose that $\lim_{\|\theta\|\to 0}  | a^\theta_{f(t,\theta)}-a(t)|=0$ for all $t\in [0,1]$ and some measurable function $a:[0,1]\to [0,1]$. Then, 
$$\lim_{\|\theta\|\to 0} \sum_{m=1}^{f(t,\theta)} \theta_m a^\theta_m=\int_0^t a(s)ds\,.$$
\end{lemme}

The proof of this result is omitted. 


\paragraph{Proof of Proposition~\ref{limittheta}.}
For each $m\geq 1$, let $Q_m^\theta\in \R^{n\times n}$ be the transition matrix induced by 
$(\si^\theta,\tau^\theta)$ at stage $m$. By the definition of these strategies, there exist coefficients $c_\ell \geq 0$ and exponents $e_\ell\geq 0$ for all $1\leq \ell\leq n$ so that
$$(Q_m^\theta)^{k,\ell}=\begin{cases} c_\ell (\la^\theta_m)^{e_\ell}+o ((\la^\theta_m)^{e_\ell })& \text{ if } \ell \neq k\\ 
1- c_k (\la^\theta_m)^{e_k }+o((\la^\theta_m)^{e_k })& \text{ if } \ell=k\,.\end{cases}$$
Moreover one can assume without loss of generality that $e_\ell=0$ whenever $c_\ell=0$ for all $\ell\neq k$, so that $e_k\geq e_\ell$ for all $\ell\neq k$. 
%
 Let $p(t,\theta)$ be the probability of being at $k$ after $\varphi(t,\theta)$ stages. Then
$$ p(t,\theta)= \prod_{m=1}^{\varphi(t,\theta)}(Q_m^\theta)^{k,k}=
\prod_{m=1}^{\varphi(t,\theta)}\left(1-c_k (\la^\theta_m)^{e_k}+ o((\la^\theta_m)^{e_k })\right)\,.$$ 
Taking $\|\theta\|$ to $0$ and setting $p(t):=\lim_{\|\theta\|\to 0} p(t,\theta)$ one has, by Lemma \ref{tecnico1},
$$p_t=\lim_{\|\theta\|\to 0} \exp\left(-c_k \sum_{m=1}^{\varphi(t,\theta)} (\la^\theta_m)^{e_k} \right)= \begin{cases}
0 & \text{ if } e_k<1\\
(1-t)^c & \text{ if }  e_k=1\\
1 & \text{ if } e_k>1\,.
\end{cases}  $$
In other words, $p_t$ is well-defined and does not depend on the family of vanishing evaluations. Similarly, conditional on reaching an absorbing at stage $m+1$, the probability that $k_{m+1}=\ell$ is given by 
$$\PP_{\si^\theta,\tau^\theta}^k\left(k_{m+1}=\ell \, | \, k_{m}=k\right)=
\frac{(Q_m^\theta)^{k,\ell}}{\sum_{\ell'\neq k} (Q_m^\theta)^{k,\ell'}}= \frac{c_\ell (\la^\theta_m)^{e_\ell }+ o ((\la^\theta_m)^{e_\ell })}{c_k(\la^\theta_m)^{e_k}+o((\la^\theta_m)^{e_k})}\,.$$
Thus, the limits as $\|\theta\|$ goes to $0$ exist for all $\ell\neq k$, and do not depend on the family of vanishing evaluations. Let $a_\ell:=\frac{c_\ell }{c_k}\ind_{\{e_\ell=e_k\}}$ denote this limit. 
Together, these two results imply that for all $t\in[0,1]$ the following limit exists and does not depend on the family of vanishing evaluations: 
$$P_t^{k,\ell}:=\lim_{\|\theta\|\to 0} \PP_{\si^\theta,\tau^\theta}^k\left(k_{\varphi(\theta,t)}=\ell\right)=
p_t \ind_{\{\ell=k\}} + (1-p_t)a_\ell\ind_{\{\ell\neq k\}} \qquad \forall 1\leq \ell \leq n\,.$$
Hence, the same is true for the cumulated times $\Pi_t^{k,\ell}:=\int_0^t P_s^{k,\ell} ds$. 
Finally, $\lim_{\|\theta\|\to 0} g(\si^\theta,\tau^\theta)=g_0$ exists and does not depend on the family of vanishing evaluations either. The result follows then from Lemma \ref{convth}, since it implies in particular 
that $\Pi_t= \lim_{\|\theta\|\to 0} \sum_{m= 1}^{\varphi(\theta,t)} \theta_m \prod_{m'=1}^{m}  Q^\theta_{m'}$, so that 
$$\lim_{\la\to 0}\ga_\theta(x_\la,y_\la;t)= \Pi_t g_0\,.$$
\hfill $\blacksquare$

\subsection{Critical stochastic games}
In this section we extend the constant-payoff property to any critical stochastic game, a class that includes stochastic games satisfying $(H1)$ and $(H2)$.

A \defn{critical stochastic game} is one with the the following property.  For each $\ep>0$ there exists a family of strategies $(x_\la^\ep,y_\la^\ep)$ in the discounted stochastic game which satisfies
\begin{itemize}
\item $(x_\la^\ep)$ is asymptotically $\ep$-optimal for Player~1 and $(y_\la^\ep)$ is asymptotically $\ep$-optimal strategy for Player~2, that is for any $(\si,\tau)\in \Sigma \times \mathcal{T}$, 
$$\liminf_{\la\to 0}\ga_\la^k(x_\la^\ep,\tau)\geq v^k-\ep\quad \text{ and }\quad \limsup_{\la\to 0}\ga_\la^k(\si, y_\la^\ep)\leq v^k+\ep\,.$$
\item The stochastic matrix $(Q^\ep_\la)$ on the state space induced by $(x_\la^\ep,y_\ep^\ep)$ is 
\defn{critical}, that is, there exists a transition matrix $M\in \R^{n\times n}$ (of a continuous-time Markov chain, so that for all $1\leq \ell \leq n$, $M^{\ell,\ell'}\geq 0$ for all $\ell'\neq \ell$, and $M^{\ell,\ell}+\sum_{\ell'\neq \ell} M^{\ell,\ell'}=0$) so that 
$$Q^\ep_\la=\Id + M \la+o(\la)\,.$$
\end{itemize}

\begin{proposition} Every stochastic games satisfying the assumptions $(H1)$ and $(H2)$ is critical. 
\end{proposition}
\begin{proof} Let $(x_\la)$ be a family of optimal stationary strategies so that $\la\mapsto x_\la^k(i)$ admit a Puiseux expansion near $0$ for all $(k,i)\in K\times I$ and let $c(k,i)$ and $e(k,i)$ so that 
$x_\la^k(i)=c(k,i)\la^{e(k,i)}+o(\la^{e(k,i)})$. For any $T\in \N$ and $\la$ consider the strategy $\xx_\la^T$ of Player~1 defined by $(\xx_\la^T)^k(i):=0$ if $e(k,i)>0$ and otherwise 
$$(\xx_\la^T)^k(i):=\displaystyle \frac{c(k,i)T^{1-e(k,i)}\la }{\sum_{i'\in I} c(k,i')T^{1-e(k,i')}\la \ind_{e(k,i)\leq 1}}\,.$$ 
One defines a strategy $\yy_\la^T$ of Player~2 in a symmetric way. The family of Markov chains $(Q_\la^T)$ is critical by construction, so it is enough to prove that the strategies $(\xx_\la^T)$ and $(\yy_\la^T)$ are asymptotically $\ep$-optimal for $T$ large enough or, equivalently, that for any $\jj\in J^n$, 
$$\lim_{T\to +\infty} \lim_{\la\to 0} \ga_\la(\xx^T_\la,\jj)= \lim_{\la\to 0} \ga_\la(x_\la,\jj)=v\in \R^n\,.$$
This result follows from the fact that $(H1)$ implies that the $\lim_{\la\to 0} \ga_\la(x_\la,\jj)$ depends only on the exist terms of highest order for each state, so that the loss of $(\xx_\la^T)$ is bounded by the difference between the exit distributions, which tends to $0$. Similarly, $(H2)$ implies an analogue property for $(\yy^T_\la)$. 
\end{proof}

In the sequel, let $(x_\la,y_\la)$ a fixed family of optimal stationary strategies so that $\la\mapsto x^k_\la(i)$ and $\la\mapsto y_\la^k$ admit a Puiseux expansion near $0$ for all $(i,j)\in I\times J$, and let $(Q_\la)$ be the corresponding family of stochastic matrices. \textbf{We assume that $(Q_\la)$ is critical.} Let $(\si^\theta,\tau^\theta)$ be the pair of strategies indexed by $\theta$ defined in Section \ref{mainOB15}, and for each $m\geq 1$ let $Q_m^\theta\in \R^{n\times n}$ be the transition matrix induced by $(\si^\theta,\tau^\theta)$ at stage $m$. By the choice of $(x_\la,y_\la)$ for each $1\leq \ell,\ell'\leq n$ and $m\geq 1$, there exist $c(\ell,\ell')\geq 0$ and $e(\ell,\ell')\geq 0$ so that 
$$Q^\theta_m(\ell,\ell')= c(\ell,\ell')(\la^\theta_m)^{e(\ell,\ell')}\,.$$ 

\paragraph{A family of continuous-time processes indexed by $\theta$.}
First of all, the family $(Q_\la)$ being critical, every state is a payoff-relevant cycle. 
Fix an initial state $1\leq k\leq n$. 
Let $(Y^{k,\theta}_m)_{m\geq 1}$ be the random process of states $(k_m)$ under the law $\PP^k_{\si^\theta,\tau^\theta}$, which is a inhomogeneous Markov chain with transition matrices $(Q^\theta_m)$. For any $t,h\geq 0$ so that $0\leq t\leq t+h\leq 1$, let $J^\theta_{[t,t+h]}$ be the number of jumps (i.e. changes of state) of the process $(Y^{k,\theta}_m)_{m\geq 1}$ in the interval $[\varphi(t,\theta),\varphi(t+h,\theta)]$. Finally, let $(X^{k,\theta}_t)$ be the time-changed process defined on $[0,1]$ by 
$$X^{k,\theta}_t:= Y^{k,\theta}_{\varphi(t,\theta)}\qquad \forall t\in[0,1]\,.$$
\paragraph{Notation.} In the sequel, we will use the following notation. 
\begin{itemize}
\item For any $t,h\geq 0$ so that $0\leq t\leq t+h\leq 1$ define 
$$P_{t,t+h}^\theta:=\prod_{m=\varphi(t,\theta)}^{\varphi(t+h,\theta)}Q_m^\theta\,.$$
\item For all $t\in[0,1]$ and $1\leq \ell \leq n$, let $\PP^\ell_t$ denote the conditional probability on $\{X_t^{k,\theta}=\ell\}$, so that 
for all $1\leq \ell, \ell'\leq n$ and $0\leq t\leq t+h\leq 1$,
 $$\PP^\ell_t(X^{k,\theta}_{t+h}=\ell'):= \PP(X^{k,\theta}_{t+h}=\ell'\,|\, X_t^{k,\theta}=\ell)=(P_{t,t+h}^\theta)^{\ell,\ell'}\,$$
 \end{itemize}

\begin{proposition}\label{lp} Let $(Q_\la)_\la$ be so that $Q_\la=\Id + A \la+o(\la)$ for a transition matrix $A\in \in \R^{n\times n}$. 
  Then, for all $t\in [0,1)$, 
\begin{itemize}
 \item[$(i)$] $\lim_{\|\theta\|\to 0} \PP^\ell_t(X^{k,\theta}_{t+h}=\ell)=\displaystyle 1+\frac{A^{\ell,\ell} }{1-t}h + o(h)$. 
 \item[$(ii)$] $\lim_{\|\theta\|\to 0} \PP^\ell_t(X^{k,\theta}_{t+h}=\ell')=\displaystyle \frac{A^{\ell,\ell'}}{1-t}h + o(h)$.
\end{itemize}
\end{proposition}
\begin{proof} 
\textbf{$(i)$}
Conditional to $\{X^{k,\theta}_{t}= \ell\}$, the event $\{X^{k,\theta}_{t+h}= \ell\}$ is the disjoint union of 
$\{J^\theta_{[t,t+h]}=0\}$ and $\{X^{k,\theta}_{t+h}= \ell\}\cap \{J^\theta_{[t,t+h]}\geq 2\}$.
For the former, one has 
\begin{equation}\label{az}
 \lim_{\|\theta\|\to 0} \PP^\ell_t(J^\theta_{[t,t+h]}=0)=1 + \frac{A^{\ell,\ell}}{1-t}h+o(h)\,.
 \end{equation}
Indeed, 
\begin{eqnarray*}
 \lim_{\|\theta\|\to 0} \PP^\ell_t(J^\theta_{[t,t+h]}=0)&=& \lim_{\|\theta\|\to 0} \prod_{m=\varphi(t,\theta)}^{\varphi(t+h,\theta)}
 \PP^\ell_t(X^{k,\theta}_{m+1}=X^{k,\theta}_m),\\
 &=& \lim_{\|\theta\|\to 0} \prod_{m=\varphi(t,\theta)}^{\varphi(t+h,\theta)} \left(1-\sum_{s'\neq \ell} (Q^\theta_m)^{\ell,\ell'}\right),\\
 &=& \lim_{\|\theta\|\to 0} \prod_{m=\varphi(t,\theta)}^{\varphi(t+h,\theta)} \left(1-\la^\theta_m |A^{\ell,\ell}|+o(\la^\theta_m)\right),\\
    &=& \lim_{\|\theta\|\to 0} \exp\left(-|A^{\ell,\ell}|\sum_{m=\varphi(t,\theta)}^{\varphi(t+h,\theta)}\la^\theta_m \right),\\
\end{eqnarray*}
and the result follows from Lemma \ref{tecnico1}. For the latter, namely $\{X^{k,\theta}_{t+h}= \ell\}\cap \{J^\theta_{[t,t+h]}\geq 2\}$, one has 
\begin{equation}\label{2jumps}\PP^\ell_t(J^\theta_{[t,t+h]}\geq 2)\leq \max_{1\leq \ell'\leq n} \PP^t_{\ell'}(J^\theta_{[0,h]}\geq 1)^2=\max_{1\leq \ell'\leq n}\left(1-\PP^t_{\ell'}(J^\theta_{[t,t+h]}=0)\right)^2\,.
\end{equation}
Therefore, $\lim_{\|\theta\|\to 0} \PP^\ell_t(J^\theta_{[t,t+h]}\geq 2)=o(h)$, which together with \eqref{az} proves the desired result. \\
\textbf{$(ii)$} Similarly, conditional on $\{X^{k,\theta}_{t}= \ell\}$, 
$$\{X^{k,\theta}_{t+h}=\ell'\}= \{X^s_{t+h}=\ell'\} \cap \left(\{ J^\theta_{[t,t+h]}=1\} \cup  \{J^\theta_{[t,t+h]}\geq 2\} \right)\,.$$
Together with \eqref{2jumps} this equality yields
$$\lim_{\|\theta\|\to 0} \PP^\ell_t(X^{k,\theta}_{t+h}=\ell')=\lim_{\|\theta\|\to 0} \PP^\ell_t(J^\theta_{[t,t+h]}=1,\, X^{k,\theta}_{t+h}=\ell')+o(h)\,.$$
Conditional on leaving the state $\ell$ at stage $m$, the probability of going to $\ell'\neq \ell$ is given by 
$$\PP(X^{k,\theta}_{m+1}=\ell'\,|\, X^{k,\theta}_m= \ell, X^{k,\theta}_{m+1}\neq \ell)=\frac{(Q^\theta_m)^{\ell,\ell'}}{\sum_{\ell''\neq \ell} (Q^\theta_m)^{\ell,\ell''}}\,.$$
By assumption, this converges to $\frac{A^{\ell,\ell'}}{|A^{\ell,\ell}|}$ as $\|\theta\|$ goes to $0$.
On the other hand, \eqref{2jumps} implies $$\lim_{\|\theta\|\to 0} \PP^\ell_t(J^\theta_{[t,t+h]}=1)=\lim_{\|\theta\|\to 0} 1-\PP^\ell_t(J^\theta_{[t,t+h]}=0)+o(h)\,.$$
Consequently, one has 
\begin{eqnarray*}\label{ert}\lim_{\|\theta\|\to 0} \PP^\ell_t(J^\theta_{[t,t+h]}=1,\, X^{k,\theta}_{t+h}=\ell')&=&\lim_{\|\theta\|\to 0}
\frac{(Q^\theta_m)^{\ell,\ell'}}{\sum_{\ell''\neq \ell} (Q^\theta_m)^{\ell,\ell''}}\left(1-\PP^\ell_t(J^\theta_{[t,t+h]}=0)+o(h)\right)\\ \label{az2}
&=&\frac{A^{\ell,\ell'}}{|A^{\ell,\ell}|}\left(\frac{|A^{\ell,\ell}|}{1-t}h+o(h)\right)\\
&=&\frac{A^{\ell,\ell'}}{1-t}h+o(h),
\end{eqnarray*}
where we used \eqref{az} to deduce the second equality.  
\end{proof}
\begin{corollaire}\label{limitOB15} The processes $(X^{k,\theta}_{t})_{t \in [0,1]}$ converge, as $\theta$ tends to $0$, to a inhomogeneous Markov process 
with generators $\left(\frac{1}{1-t}A\right)_{t\in[0,1]}$.
\end{corollaire}
\begin{proof}
 The limit is identified by Proposition \ref{lp}. The
 tightness is a consequence of $(ii)$. Indeed, it implies that for any $T>0$,
 uniformly in $\theta$:
 $$\lim_{\ep\to 0} \PP\left(\exists t_1,t_2\in [0,T] \, | \, t_1<t_2<t_1+\ep,\ X^{k,\theta}_{t_1^-}\neq X^{k,\theta}_{t_1},\ X^{k,\theta}_{t_2^-}\neq X^{k,\theta}_{t_2}
 \right)=0,$$
which is precisely the tightness criterion for c\`adl\`ag process with discrete values.
 \end{proof}
\begin{corollaire}\label{Pitheta} For all $t\in [0,1]$ the following limit exist: 
$$\Pi_t:= \lim_{\|\theta\|\to 0} \sum_{m= 1}^{\varphi(\theta,t)} \theta_m \prod_{m'=1}^{m}  Q^\theta_{m'} = \int_{0}^t e^{-\ln(1-s)A} ds\,.$$
\end{corollaire}
Corollary \ref{Pitheta} follows from Corollary~\ref{limitOB15} together with Lemma \ref{convth}. 



\bibliographystyle{amsplain} 
\bibliography{bibliothese2} 

\end{document}